\documentclass[
11pt,
letterpaper,
onecolumn,
oneside]{article}

\usepackage[margin=1 in,letterpaper]{geometry}
\usepackage[singlespacing]{setspace}
\usepackage{fancyhdr}
\usepackage[inline]{enumitem}
\usepackage{newclude}
\usepackage{caption}
\usepackage{subcaption}
\usepackage[utf8]{inputenc}
\usepackage[T1]{fontenc}
\usepackage{lmodern}
\usepackage{appendix}
\usepackage{comment}
\usepackage{graphicx}
\usepackage[colorlinks=true,citecolor=black,linkcolor=black,urlcolor=black]{hyperref} 
\usepackage{amssymb}

\usepackage{amsthm}
\usepackage{mathtools}
\usepackage{bm}
\usepackage{mathrsfs}
\usepackage{bbm}
\usepackage{xcolor}

\theoremstyle{definition}
\newtheorem{definition}{Definition}[section]
\newtheorem{remark}[definition]{Remark}

\theoremstyle{plain}
\newtheorem{theorem}[definition]{Theorem}
\newtheorem{proposition}[definition]{Proposition}
\newtheorem{lemma}[definition]{Lemma}

\DeclareMathOperator{\aut}{aut}

\newcommand{\Pb}{{\mathbb P}}
\newcommand{\Eb}{{\mathbb E}}
\newcommand{\eps}{\varepsilon} 

\newcommand{\mc}{\mathcal}

\usepackage{xcolor}

\newcommand{\fgraph}{{H_F}}

\begin{document}
\title{Sharp Thresholds for Factors in Random Graphs}
\author{Fabian Burghart\thanks{Eindhoven University of Technology, Department of Mathematics and Computer Science, MetaForum MF 4.091, 5600 MB Eindhoven, the Netherlands. Email: \href{mailto:f.burghart@tue.nl}{\nolinkurl{f.burghart@tue.nl}}. This research has received funding from the European Union's Horizon 2020 research and innovation programme under the Marie Sk\l{}odowska-Curie Grant Agreement no. 101034253.}
    \and
    Annika Heckel\thanks{Matematiska institutionen, Uppsala universitet, Box 480, 751 06 Uppsala, Sweden. Email: \href{annika.heckel@math.uu.se}{\nolinkurl{annika.heckel@math.uu.se}}. Funded by the Swedish Research Council, Starting Grant 2022-02829.}. 
    \and
    Marc Kaufmann\thanks{Institut für Theoretische Informatik, ETH Zürich, Zürich, Switzerland. Email: \href{mailto:marc.kaufmann@inf.ethz.ch}{\nolinkurl{marc.kaufmann@inf.ethz.ch}}. The author gratefully acknowledges support by the Swiss National Science Foundation [grant number 200021\_192079].}
    \and
    Noela Müller\thanks{Eindhoven University of Technology, Department of Mathematics and Computer Science, MetaForum MF 4.084, 5600 MB Eindhoven, the Netherlands. Email: \href{mailto:n.s.muller@tue.nl}{\nolinkurl{n.s.muller@tue.nl}}. Research supported by the NWO Gravitation project NETWORKS under grant no. 024.002.003. }
    \and
    Matija Pasch\thanks{Email: \href{mailto:matija.pasch@gmail.com}{\nolinkurl{matija.pasch@gmail.com}}}
}
\maketitle

\begin{abstract}
 Let $F$ be a graph on $r$ vertices and let $G$ be a graph on $n$ vertices.
 Then an $F$-factor in $G$ is a subgraph of $G$ composed of $n/r$ vertex-disjoint copies of $F$, if $r$ divides $n$. In other words, an $F$-factor yields a partition of the $n$ vertices of $G$.
 The study of such $F$-factors in the Erdős–Rényi random graph dates back to Erdős himself.
 Decades later, in 2008, Johansson, Kahn and Vu established the thresholds for the existence of an $F$-factor for strictly 1-balanced~$F$ -- up to the leading constant.
 The sharp thresholds, meaning the leading constants, were obtained only recently by Riordan and Heckel, but only for complete graphs $F=K_r$ and for so-called \emph{nice} graphs.
 Their results rely on sophisticated couplings that utilize the recent, celebrated solution of Shamir's problem by Kahn.

 We extend the couplings by Riordan and Heckel to any strictly 1-balanced $F$ and thereby obtain the sharp threshold for the existence of an $F$-factor.
 In particular, we confirm the thirty year old conjecture by Rucínski that this sharp threshold indeed coincides with the sharp threshold for the disappearance of the last vertices which are not contained in a copy of $F$.\\
 
 \noindent \emph{2020 Mathematics Subject Classification.} 05C70 (primary); 05C80, 60C05 (secondary).
 
 \noindent \emph{Key words and phrases.} Sharp threshold; graph factor; 1-balanced graphs. 
\end{abstract}

\section{Introduction}\label{sec:introduction}
What is the simplest, local obstruction for a graph $G$, on an even number of vertices, to admit a perfect matching? Well, if an isolated vertex exists, a perfect matching cannot.
Back in 1966, Erdős and Rényi \cite{erdos1966existence} showed that this simplest, local obstruction is indeed the only relevant obstruction, when it comes to the \emph{random graph} $G(n,p)$, with vertices $[n]=\{1,\dots,n\}$, and where each edge is included independently with probability $p$.
To be precise, recall that a sequence $p^*=p^*_n$ is a \emph{sharp threshold} for a graph property $\mc E=\mc E_n$ if for all $\eps>0$ and $p\le(1-\eps)p^*$ the random graph $G(n,p)$ does not have $\mc E$ with high probability (whp), that is, with probability tending to one as $n$ increases, while for all $p\ge(1+\eps)p^*$ the random graph $G(n,p)$ does have $\mc E$ whp.
Erdős and Rényi \cite{erdos1966existence} observed that the sharp threshold for the existence of perfect matchings coincides with the sharp threshold for the disappearance of isolated vertices.

Now, fix a connected graph $F$ with $e(F)=s$ edges on $v(F)=r\ge 2$ vertices.
Further, let $G$ be a graph on $n$ vertices.
An \emph{$F$-factor} in $G$ is a collection of vertex-disjoint copies of $F$ in $G$ that cover all $n$ vertices. Thus, there are $n/r$ copies and hence $r$ necessarily divides $n$, which we assume in the following without further mention.
For example, in the base case where $F$ is an edge, an $F$-factor in $G$ is just a perfect matching, as discussed above (where we required $n$ to be even).

Can we also find a simple, local obstruction for the existence of a general $F$-factor?
Well, if there exists a vertex $v$ that is not contained in any copy of $F$ in $G$, then $G$ cannot have an $F$-factor. In analogy to the above, let the \emph{$F$-degree} of $v$ be the number of copies of $F$ in $G$ that contain $v$, and let $v$ be \emph{$F$-isolated} if its $F$-degree is zero.

The sharp threshold $p^*$ for the disappearance of $F$-isolated vertices is well-understood \cite{janson2000random,rucinski1992matching,spencer1990threshold}.
However, its discussion is rather cumbersome, so we content ourselves with building some intuition. The expected $F$-degree of a fixed vertex $v$ in $G(n,p)$ is
of order $n^{r-1}p^{s}=(np^{d_1(F)})^{r-1}$, where
\begin{align}\label{equ:1density_def}
 d_1(F)=\frac{e(F)}{v(F)-1}
\end{align}
is the \emph{1-density} of $F$.
In analogy to the discussion for graphs (and uniform hypergraphs, for that matter), we would expect the $F$-isolated vertices to disappear when the expected $F$-degree is of order $\ln(n)$
yielding that $p^*$ should be of order $\ln(n)^{1/s}/n^{1/d_1(F)}$.
However, what's missing in this line of reasoning, is that also for any subgraph $S\subset F$ the expected average degree cannot be too small. Say, if there exists $S\subset F$ with $d_1(S)>d_1(F)$, we quickly start to furrow our brows:
The expected $S$-degree is significantly smaller than the expected $F$-degree.
But whenever all $F$-isolated vertices disappear, also at least some linear number of $S$-isolated vertices disappears.
Indeed, in this case we would have to dive a lot deeper into the matter \cite[Chapter 3]{janson2000random}. But if $F$ is \emph{strictly 1-balanced}, that is, we have $d_1(S)<d_1(F)$ for any non-trivial proper subgraph $S\subsetneqq F$, then our intuition turns out to be correct.
Examples of such graphs are cycles and cliques, and more can be found in \cite{johansson2008factors,janson2000random}.
Let $\aut(F)$ be the number of automorphisms of $F$.
\begin{theorem}\label{thm:threshold_COV}
Let $F$ be a strictly 1-balanced graph with $s>0$ edges on $r\ge 2$ vertices.
Then the sharp threshold for the disappearance of $F$-isolated vertices in $G(n,p)$ is
$$p^*=\left(\frac{\aut(F)}{r!}\cdot\frac{\ln(n)}{\binom{n-1}{r-1}}\right)^{1/s}.$$
\end{theorem}
\begin{proof}
This is an application of Theorem 3.22 (i) in \cite{janson2000random}, originally due to \cite{rucinski1992matching}.
\end{proof}
With the sharp threshold $p^*$ for the disappearance of $F$-isolated vertices in place, we turn back to the existence of $F$-factors.

Reportedly \cite{alon1993threshold}, Erdős was already concerned with the $F$-factor threshold for specific $F$.
Łuczak and Ruciński \cite{LuczakRucinski1987,LuczakRucinski1991} observed around 1990 that the sharp thresholds for the existence of $F$-factors and the disappearance of $F$-isolated vertices coincide whenever $F$ is a tree.
In 1992, Ruciński \cite{rucinski1992matching} bounded the $F$-factor threshold for a certain class of $F$. Further, he conjectured that the two thresholds coincide for all strictly 1-balanced $F$.
Subsequently, Alon and Yuster \cite{alon1993threshold} and Krivelevich \cite{Krivelevich1997TriangleFI} improved the bounds for $F$-factor thresholds for different classes of $F$.
Next to a very elaborate discussion of both thresholds, in 2000, Janson, Łuczak and Ruciński \cite{janson2000random} provided an example where the two thresholds do not coincide and stated the weaker conjecture that the two thresholds coincide for complete graphs.
After another improvement by Kim \cite{kim2003} in 2003,
Johansson, Kahn and Vu \cite{johansson2008factors} established in a celebrated contribution from 2008 that the two thresholds are of the same order for strictly 1-balanced $F$.
They also obtained an equally impressive result for more general $F$.
Next to a conjecture for the general case, they also implicitly conjectured that the sharp $F$-factor threshold coincides with $p^*$ for strictly 1-balanced~$F$.

After a partial confirmation of the general conjecture in \cite{johansson2008factors} by Gerke and McDowell \cite{gerke2015nonvertex} in 2014 for a certain class of nonvertex-balanced graphs, Riordan \cite{riordan2022random} and Heckel \cite{heckel2020random} established by 2022 that the two sharp thresholds coincide for complete graphs $F$, thereby resolving the weaker conjecture in \cite{janson2000random}.
\begin{theorem}\label{thm:threshold_Kr}
 Let $F$ be the complete graph on $r\ge2$ vertices.
 Then the sharp threshold for the existence of an $F$-factor is $p^*=p^*(F)$ from Theorem \ref{thm:threshold_COV}.
\end{theorem}
Riordan \cite{riordan2022random} further established the threshold for a more general class of hypergraphs. We only state the result for the graph case.
\begin{theorem}\label{thm:threshold_nice}
 Let $F$ be a strictly 1-balanced, 3-vertex-connected graph such that no graph obtained by removing one edge from $F$ and adding another is isomorphic to $F$.
 Then the sharp threshold for the existence of an $F$-factor in $G(n,p)$ is $p^*=p^*(F)$ from Theorem \ref{thm:threshold_COV}.
\end{theorem}
Heckel and Riordan obtained Theorem \ref{thm:threshold_Kr} and Theorem \ref{thm:threshold_nice} by constructing a coupling of $G(n,p)$ and the \emph{random $r$-uniform hypergraph} $H_r(n,\pi)$, with vertices $[n]$, and where each hyperedge is included independently with probability $\pi$. Then they utilized the following celebrated solution to Shamir's problem by Kahn \cite{kahn2023asymptotics,kahn2022hitting} to derive the existence of an $F$-factor in $G(n,p)$ from the existence of a perfect matching in $H_r(n,\pi)$.
\begin{theorem}\label{thm:shamirs_problem}
 For any $r\ge 2$, the sharp threshold for the existence of a perfect matching in $H_r(n,\pi)$ is
 \[
  \pi^*=\frac{\ln(n)}{\binom{n-1}{r-1}}.
 \]
\end{theorem}
In one of the two contributions,  Kahn \cite{kahn2023asymptotics} also conjectures that the sharp $F$-factor threshold is $p^*=p^*(F)$ from Theorem \ref{thm:threshold_COV} for all strictly 1-balanced $F$ and slightly beyond.
Finally, in a groundbreaking contribution \cite{park2024proof} by Park and Pham from 2024, the expectation-threshold conjecture by Kahn and Kalai \cite{kahn2007thresholds} was confirmed. While not relevant to our case at hand, this result has significant implications for the general case and the discussion in \cite{johansson2008factors} in particular.

\medskip

\noindent\textbf{Hitting Times.}
While the results above are impressive to say the least, they still fail to capture just how deep the connection of $F$-factors and $F$-isolated vertices actually is.
This deep connection is captured by the \emph{hitting-time versions} of some of the results above \cite{bollobas1985,Ajtai1987FirstOccurrence,LuczakRucinski1991, kahn2023asymptotics,kahn2022hitting,heckel2024hitting,burghart2024}.

Restricted to the graph case, these refinements state that the \emph{hitting times} for the disappearance of the last $F$-isolated vertex and the existence of an $F$-factor coincide whp in the standard random graph process, with vertices $[n]$, where we start with the empty graph $G_0$ and include a new edge uniformly at random in each time step $t$ to obtain the graph $G_t$. So, if we let $T$ denote the step where the last $F$-isolated vertex disappears, then we already find an $F$-factor in $G_T$ whp.
These results confirm in a very strong sense that the simple, local obstruction given by $F$-isolated vertices is indeed typically the only reason for the absence of $F$-factors.

\subsection{Main Result}\label{sec:main_result}
The purpose of this article is to extend Theorem \ref{thm:threshold_Kr} and Theorem \ref{thm:threshold_nice} (in the graph case), unify them and to confirm the strong conjecture in \cite{rucinski1992matching}, thereby partially confirming the conjectures from \cite{johansson2008factors,kahn2023asymptotics}.
Recall $d_1(F)$ from Equation \eqref{equ:1density_def} and that a graph $F$ on $r\ge 2$ vertices is strictly 1-balanced if $d_1(S)<d_1(F)$ for all non-trivial proper subgraphs $S\subsetneqq F$. Also, recall the sharp threshold $p^*=p^*(F)$ for the disappearance of the last $F$-isolated vertex in the random graph $G(n,p)$ from Theorem \ref{thm:threshold_COV}.
\begin{theorem}\label{thm:main_threshold}
Let $F$ be a strictly 1-balanced graph with $s>0$ edges on $r\ge 2$ vertices.
Then the sharp threshold for the existence of an $F$-factor in the random graph $G(n,p)$ is $p^*$.
\end{theorem}
As in \cite{heckel2020random,riordan2022random},
we construct a coupling and then apply Theorem \ref{thm:shamirs_problem} to obtain Theorem \ref{thm:main_threshold} as a corollary.
However, the coupling has several immediate implications, which is why we state it as a second main result. For this purpose, we recall the following definitions from~\cite{riordan2022random}.

An \emph{$F$-graph} $H=(V,E)$ is given by the vertices $V$ and \emph{$F$-edges} $e\in E$, where each $F$-edge $e$ is a copy of $F$ with vertices in $V$.
The \emph{random $F$-graph} $\fgraph(n,\pi)$ has the vertex set $[n]$ and each of the $\binom{n}{r} \frac{r!}{\mathrm{aut}(F)}$ potential copies of $F$ with vertices in $[n]$ is included independently with probability $\pi$.
\begin{theorem}\label{thm:main_coupling}
Let $F$ be a strictly 1-balanced graph with $s>0$ edges on $r\ge 2$ vertices.
Then there exist $\delta,\eps>0$ such that the following holds.
For any sequences $p\le n^{-1/d_1(F)+\eps}$ and $\pi\le(1-n^{-\delta})p^{e(F)}$ there exists a coupling of $G=G(n,p)$ and $H=\fgraph(n,\pi)$ such that whp, for each $F$-edge in $H$ the corresponding copy of $F$ is present in $G$.
\end{theorem}
This coupling allows to transfer results for the random $r$-uniform hypergraph $H_r(n,\pi)$ to the random graph $G(n,p)$, not only the existence of perfect matchings \cite{heckel2020random,riordan2022random}, but also $k$-connectivity \cite{Poole2015StrengthConnectedness} or the existence of loose Hamilton cycles \cite{frieze2010loose,dudek2010loose,dudek2012optimal,frieze2020,diaz2023spanning}.
The required coupling of $H_r(n,\pi')$ and $\fgraph(n,\pi)$ is given by including a hyperedge $h$ in $H_r(n,\pi')$ if there exists an $F$-edge in $\fgraph(n,\pi)$ with vertex set $h$, which yields the inclusion probability
\begin{equation}\label{eq:HFHr}
 \pi'=1-(1-\pi)^{r!/\aut(F)}.
\end{equation}

\begin{remark}
 It follows from the proof that Theorems~\ref{thm:main_threshold} and \ref{thm:main_coupling} remain valid if $F$ is a strictly 1-balanced $u$-uniform hypergraph. In this case, $G(n,p)$ in these theorems needs to be replaced by $H_u(n,p)$, the random $u$-uniform hypergraph. In fact, the proof for $u>2$ can be simplified, since two copies of $F$ overlapping in exactly two vertices can never intersect in an $u$-edge. Thus, the obstacle posed by sparse clean cycles as defined in Section~\ref{sec:fps:dgraphs} does not occur if $u>3$, and we can couple $H_u(n,p)$ directly to $H_F(n,\pi)$ without requiring a hypergraph version of the d-graph $G^*$ from Definition~\ref{def:dgraph}.  
\end{remark}

\section{Proof Strategy}\label{sec:proof_strategy}
The proof of Theorem \ref{thm:main_coupling}
follows the line of reasoning in \cite{heckel2020random}. However, due to the extension of the very specific case of the triangle $K_3$ to the very general class of strictly 1-balanced graphs, we introduce a few major modifications and replace several key arguments. In the following high-level presentation of our proof, we will explicitly point out these novelties.

Both the proof in \cite{heckel2020random} and our discussion are modifications of the beautiful proof in \cite{riordan2022random}.
In particular, we borrow concepts as well as results from \cite{riordan2022random} that directly apply to general strictly 1-balanced graphs.
Thus, we will also compare our discussion to \cite{riordan2022random}.
In order to convey the relevant ideas underlying the proof, we will use intuitive descriptions of the central concepts and postpone the technical formal definitions to the next section.

Finally, recall from Section \ref{sec:main_result} that Theorem \ref{thm:main_threshold} follows from Theorem \ref{thm:main_coupling} and Theorem~\ref{thm:shamirs_problem}.

\subsection{Induced Edges, Avoidable Configurations and Clean Cycles}\label{sec:ps:induced}
When trying to construct a coupling between the random $F$-graph $H=\fgraph(n,\pi)$ and the copies of $F$ in the random graph $G=G(n,p)$, one invariably encounters the following fundamental problem.
Let $H'$ be the $F$-graph of $G$, meaning that we include each copy of $F$ in $G$ as an $F$-edge in $H'$.
For all $0<\pi<1$, $H$ can clearly attain any $F$-graph.
But this is not the case for $H'$, due to the existence of so-called 'induced $F$-edges': 
Assume that there is an $F$-edge $h$, meaning a copy $h$ of $F$, which is not contained in $H$, but that for each (graph) edge $e$ in the copy $h$ of $F$ there exists an $F$-edge $h_e$ in $H$ that covers $e$.
Then $H'$ cannot attain $H$, because if the copies $h_e$ of $F$ appear in $G$, then so does the copy $h$ of $F$.

Riordan shows that only two types of (minimal) $F$-graphs can induce $F$-edges: so-called `avoidable configurations' and so-called `clean $F$-cycles'.
Roughly speaking, the former are given by connected $F$-graphs which have at least two cycles and the latter
consist of exactly one cycle, with `cycles' defined such that two hyperedges that overlap on more than two vertices have more than one cycle.
He proceeds to show that avoidable configurations do not exist whp in $H$.

\subsection{Sparse Clean Cycles and Dummy Edges}\label{sec:ps:sparse_cc}
This leaves us with clean $F$-cycles, which may exist in $H$.
Riordan avoids the discussion of these clean $F$-cycles by restricting to the class of graphs $F$ from Theorem \ref{thm:threshold_nice}, which he called `nice': A graph with these properties cannot be induced by a clean $F$-cycle.
Heckel \cite{heckel2020random}, on the other hand, tweaks Riordan's coupling to deal with the only  type of clean $K_3$-cycle which can induce a triangle $K_3$, namely the one of length $3$. 
Since we consider the case of general strictly 1-balanced graphs, we also have to deal with clean $F$-cycles. But as opposed to Heckel, we have to consider all types of clean $F$-cycles. While this is a major complication in its own right, there is one particular type of clean $F$-cycle which is particularly demanding: so-called `sparse clean $F$-cycles', which are given by two $F$-edges, meaning two copies of $F$, that share exactly one (graph) edge (and no other vertices), as in Figure~\ref{fig:sparseclean}.

\begin{figure}
 \centering
 \includegraphics[scale=1]{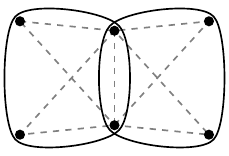}
 \caption{A sparse clean $F$-cycle, where $F$ is the graph obtained from $K_4$ by deleting an edge.}
 \label{fig:sparseclean}
\end{figure}

Recall that we have $\pi\approx p^{e(F)}$.
So a clean $F$-cycle $C$ of length $k$ and its shadow graph $C'$, which is obtained by taking the graph union of the $F$-edges in $C$, have the same inclusion probability $\pi^k\approx p^{ke(F)}$, in $H$ and $G$ respectively --- unless $C$ is a sparse clean $F$-cycle, because then we are missing an edge in $C'$. This small specialty of $C$ has significant implications for the entire proof.
The treatment of these sparse clean $F$-cycles is one of our two main contributions.

Heckel heavily relies on the fact that the expected number of clean $3$-cycles in $H$ (the only relevant type of clean $K_3$-cycle in the setting of \cite{heckel2020random}) matches the expected number of their shadow graphs in $G$ almost exactly. 
This allowed her to couple the clean $3$-cycles in $H$ and their shadow graphs in $G$ so that they match exactly whp. 
For sparse clean $F$-cycles, these expected numbers differ by a factor of $p$, thus this approach is doomed to fail. We do the next best thing: We randomly select a subset of the shadow graphs of sparse clean $F$-cycles in $G$, where we include each such shadow graph independently with probability $p$. This ensures that the number of \emph{selected} shadow graphs matches the number of sparse clean $F$-cycles in $H$.

For our proof, it is extremely useful to model the inclusion in this random subset as follows.
For each potential copy of a sparse clean $F$-cycle, we introduce an additional `dummy edge'.
We further consider the random graph $G^*=G^*(n,p)$, where we do not only include each potential edge, but also each potential dummy edge independently with probability $p$. 
We include a shadow graph $S$ of a sparse $F$-cycle in the random subset in $G$ if $S$ \emph{and} the corresponding dummy edge is contained in $G^*$.  
In this sense, the selected shadow graph now has the right number of edges in $G^*$.
Note that for some graphs $F$, two different sparse $F$-cycles can have the same shadow graph (for example, if $F$ consists of two triangles sharing exactly one edge). In this case we consider separate dummy edges for each $F$-cycle, and a shadow graph $S$ has several `chances' to be added to the random set --- independently with probability $p$ for each potential sparse $F$-cycle that has $S$ as its shadow graph.

\subsection{Coupling Clean Cycles}\label{sec:ps:cc_coupling}
With this modification in place, we turn back to the proof in \cite{heckel2020random}, where Heckel couples the clean $F$-cycles in $H$ and their shadow graphs in $G$ so that they match up whp.
As opposed to Heckel, we have to couple \emph{all types} of clean $F$-cycles in $H$, not just one type, with their (selected) shadow graphs in $G^*$.
For this purpose, as in \cite{heckel2020random}, we have to control the overlaps of pairs of clean $F$-cycles, and the overlaps of their (selected!) shadow graphs in $G^*$.
Heckel used an exhaustive case distinction and direct computation to gain control.
This is not an option for the general case, which brings us to our second main contribution.
For the discussion of two overlapping, distinct clean $F$-cycles as well as the discussion of other more complicated $F$-graphs, we observe that these are in fact avoidable configurations. Thus, it is Riordan's work on avoidable configurations in \cite{riordan2022random} that immediately gives us control over the overlaps.
On the graph side, we establish that shadow graphs $S$ of clean $F$-cycles are strictly balanced, that is, the edge density $e(S)/v(S)$ of any proper subgraph is strictly smaller than the edge density of $S$. Here, it is yet again crucial that we consider the dummy edge when we discuss the shadow graphs of sparse clean $F$-cycles.
With this property in place, we easily control the overlaps in the graph case.
As with the avoidable configurations above, the fact that these shadow graphs are strictly balanced also helps later on in the proof.

The second step in the coupling of the clean $F$-cycles in $H$ and their (selected) shadow graphs in $G^*$ is the application of the Chen-Stein method.
Since we we need to couple all types of clean $F$-cycles and their (selected) shadow graphs \emph{jointly}, the result in \cite{ross2011} used by Heckel does not suffice for our purposes. We use a stronger version from \cite{arratia1989} which gives the desired coupling \emph{directly}, as opposed to the two-step coupling in \cite{heckel2020random} that couples the number of clean $F$-cycles in the first step and the cycles given the numbers in the second.
\subsection{The Riordan-Heckel Coupling}\label{sec:ps:heckel_riordan_coupling}
We now proceed as in \cite{heckel2020random} and introduce the Riordan-Heckel coupling, given the clean $F$-cycles in $H$ and their shadow graphs in $G^*$ (which are coupled to match up).
Here, we adapt the coupling to our auxiliary model $G^*$.
Apart from this minor modification, there is no difference in the definition of the coupling, although we do consider all types of clean $F$-cycles.

In the last step, we bound the conditional $F$-edge inclusion probabilities as in \cite{heckel2020random}, which in turn is based on \cite{riordan2022random} with a modified treatment for clean $F$-cycles.
Apart from the derivation of the initial bound in \cite[Equation 3]{riordan2022random} which is valid in general, this proof step can be split into four parts.
The discussion of $F$-edges and clean $F$-cycles which can break the coupling is similar to \cite{heckel2020random}.
The bound derived for other $F$-edges is in fact related to the combinatorial argument in \cite{burghart2024}. However, we use a different argument which is consistent with our novel approach to complete the fourth and last part: the bound for clean $F$-cycles that do not break the coupling.
As mentioned before, this result is easily derived from the property that clean $F$-cycles are strictly balanced.

To summarize, the coupling in Theorem~\ref{thm:main_coupling} is achieved in the following way: 
\[
 c_F(G(n,p)) = c_F(G^*(n,p)) \overset{\text{Riordan--Heckel}}{\supseteq} H_F(n,\pi)
\]
where we write $c_F(G)$ for the set of copies of $F$ contained in $G$. Theorem~\ref{thm:main_threshold} follows from this by observing (e.g.  that $H_F(n,\pi)$, after merging edges on the same vertex set, is distributed like $H_r(n,\pi)$, as outlined in the paragraph above equation~\eqref{eq:HFHr}. 

%
%
%
%
%
\section{Preliminaries}\label{sec:formal_proof_structure}
We turn to the formal part of the proof.
In this section, we cover the core concepts, notions and terminology.
In particular, we introduce the objects which are fixed throughout the proof in Section \ref{sec:fps:global_constants}. In Section \ref{sec:fps:Fgraphs}, we discuss $F$-graphs and related notions.
Then, in Section \ref{sec:fps:induced_edges}, we cover the induced edges, avoidable configurations and clean $F$-cycles from Section \ref{sec:ps:induced}, including the results mentioned therein. Finally, we formalize the contents of Section \ref{sec:ps:sparse_cc}, covering dummy edges and $G^*(n,p)$, in Section \ref{sec:fps:dgraphs}, where we also already establish the result mentioned in Section \ref{sec:ps:cc_coupling} that shadow graphs of clean $F$-cycles including dummy edges are strictly balanced.

The coupling of the clean $F$-cycles from Section \ref{sec:ps:cc_coupling} is then discussed in Section \ref{sec:clean_cycle_coupling}.
The coupling and the results from Section \ref{sec:ps:heckel_riordan_coupling} are established in Section \ref{sec:hypergraph_coupling}, where we complete the proof of Theorem \ref{thm:main_coupling}.
The short proof of Theorem \ref{thm:main_threshold} is given in Section \ref{sec:thm_main_threshold_proof}.
\subsection{Global Constants}\label{sec:fps:global_constants}
We fix a strictly 1-balanced graph $F$ on $r>2$ vertices. Recall that $F$ is $2$-vertex-connected \cite{riordan2022random}.
Further, we fix $\delta=\delta(F)>0$ sufficiently small and $\varepsilon=\varepsilon(F,\delta)>0$ sufficiently small. These quantities are considered constant throughout.

In the remainder, we consider the asymptotics in the number $n$ of vertices. In particular, we use the Bachmann-Landau notation with respect to $n$. Let
$$\pi\le\frac{n^{\varepsilon}}{\frac{r!}{\aut(F)}\binom{n-1}{r-1}},\quad
p=\left(\frac{\pi}{1-n^{-\delta}}\right)^{1/e(F)}$$
be the respective inclusion probabilities for $\fgraph(n,\pi)$ and $G(n,p)$.
Clearly, it is sufficient to establish the whp coupling in Theorem \ref{thm:main_coupling} for $\pi$ and $p$.
\subsection{\textit{F}-graphs}\label{sec:fps:Fgraphs}
An \emph{$F$-graph} $H=(V,E)$ is given by the vertex set $V$ and a set $E$ of $F$-edges, that is, copies of $F$ with vertices in $V$.
The \emph{shadow graph} $G=(V,E')$ of $H$ is the graph on the same vertices with edge set $E'$, which is the union of the edge sets of all $F$-edges in $H$. Formally, let $h=(V_h,E_h)$ be the $F$-edges $h\in E$, then we have
\[
 E'=\bigcup_{h\in E}E_h.
\]
Let $v(H)=|V|$ denote the number of vertices and $e(H)=|E|$ the number of edges of $H$.
Let $H$ be connected if $G$ is, and let $c(H)$ be the number of components of $H$.

For a graph $G=(V,E)$, let the \emph{$F$-graph $H=(V,E')$ of $G$} be defined via the $F$-edge set $E'$ of copies of $F$ in $G$.
Finally, note that we may think of $r$-uniform hypergraphs as $K_r$-graphs and of simple graphs as $K_2$-graphs. Thus, it suffices to introduce notions for $F$-graphs (and we take the liberty to allow $F=K_2$ for definitions).
\subsection{Induced Edges, Avoidable Configurations And Clean cycles}\label{sec:fps:induced_edges}
The following quantity from \cite{riordan2022random} is crucial to our proof.
The \emph{nullity} of an $F$-graph $H$ is
$$n(H)=(r-1)e(H)+c(H)-v(H).$$
We will tacitly use the following observations on several occasions, to identify minimizers of the nullity in classes of $F$-graphs.
Consider two connected $F$-graphs $H_1$, $H_2$.
Let $v$ be a vertex in their vertex overlap such that no $F$-edge incident to $v$ is contained in both $H_1$ and $H_2$. Obtain the copy $H'_2$ of $H_2$ by substituting $v$ with an entirely new vertex. Then we have $n(H_1\cup H'_2)=n(H_1\cup H_2)$ if $v$ is the only vertex in the overlap and $n(H_1\cup H'_2)=n(H_1\cup H_2)-1$ otherwise.

Now, let $h$ be an $F$-edge in the overlap of $H_1$ and $H_2$ and $v$ a vertex in $h$ such that each of the remaining $(r-1)$ vertices in $h$ is not incident to any other $F$-edge in the overlap. Obtain the copy $H'_2$ of $H_2$ by substituting these $(r-1)$ vertices with entirely new vertices. Then we have $n(H_1\cup H'_2)=n(H_1\cup H_2)$.

Next, let an $F$-edge $h$ be \emph{induced by $H$} if it is not contained in $H$, but it is contained in the $F$-graph of the shadow graph of $H$.
An $F$-graph $H$ is an \emph{avoidable configuration} if it is connected and $n(H)\ge 2$. Intuitively, $H$ is an avoidable configuration if it contains at least two cycles. 
Further, for $k>2$, the $F$-graph $H$ is a \emph{clean $F$-cycle of length $k$} if the following holds.
There exists an ordering $h_1,\dots,h_k$ of the $F$-edges of $H$ and there are distinct vertices $v_1,\dots,v_k$ such that $h_i$ and $h_{i+1}$ overlap exactly in $v_i$ for all $i\le k$ (using $h_{k+1}=h_1$), and $h_i$, $h_j$ do not overlap otherwise.
For $k=2$, the $F$-graph $H$ is a \emph{clean $F$-cycle of length $2$} if it has exactly two $F$-edges, which overlap in exactly two vertices. 
The next result from \cite{riordan2022random} shows that we can control induced $F$-edges by controlling avoidable configurations and clean $F$-cycles.\footnote{Note that the statement in \cite{riordan2022random} is slightly different: the conclusion is only that there exists an avoidable configuration in $H$, or a clean $F$-cycle that induces $E$. However, when following through the proof in~\cite{riordan2022random} it is clear that the avoidable configuration that is found indeed induces $E$.}
\begin{lemma}\label{lem:induced}[\cite{riordan2022random}, Le.~19]
For every induced $F$-edge $E$ of an $F$-graph $H$, there exists a sub-$F$-graph $H'\subset H$ with $e(H')\le e(F)$ that induces $E$ and which is either an avoidable configuration or a clean $F$-cycle. \qed
\end{lemma}
Note that we slightly deviate from \cite{heckel2020random,riordan2022random} as in our definition we do not introduce a bound on the size of avoidable configurations. Instead, we derive individual whp-bounds for the bounded number of avoidable configurations in our proof.
Thus, we present a slight variation of the corresponding result in \cite{riordan2022random}.
\begin{lemma}\label{lem:avoidable_configs}
Let $H$ be a fixed connected $F$-graph and $X$ the number of copies of $H$ in $\fgraph(n,\pi)$.
Then we have $\Eb[X]=\mc O(n^{1-n(H)+e(H)\eps})$.
\end{lemma}
\begin{proof}
The expected value of $X$ can be bounded by 
$$\Eb[X]=\mc O\left(n^{v(H)}\left(n^{\eps-(r-1)}\right)^{e(H)}\right)=\mc O\left(n^{-n(H)+1+e(H)\eps}\right).$$
\end{proof}
This shows that we can pick $\eps>0$ sufficiently small so that for any fixed avoidable configuration $A$, whp there are no copies of $A$ in $\fgraph(n,\pi)$.
Hence, this also holds for any bounded number of fixed avoidable configurations, which suffices for our proof.
Further, whp there is at most a small polynomial number of clean $F$-cycles of bounded length, since the nullity of a clean $F$-cycle is $1$. 
\subsection{Sparse clean cycles and dummy edges}\label{sec:fps:dgraphs}
Let $C$ be a clean $F$-cycle. We say that $C$ is \emph{sparse} if $e(C)=2$ and the two overlap vertices $v_1,v_2$ form an edge $e=\{v_1,v_2\}$ in both $F$-edges $h_1,h_2$ of $C$, that is, the edge $e$ is contained in the copy $h_1$ of $F$ as well as in the copy $h_2$ of $F$.
Otherwise, we say that the clean $F$-cycle $C$ is \emph{dense}. Thus, we have $e(C')=k e(F)$ for the shadow graph $C'$ of a dense clean $F$-cycle of length $k$, but $e(C')=2 e(F)-1$ in the sparse case.

Let $X$ be the number of copies of $C$ in $\fgraph(n,\pi)$.
Further, let $Y$ be the number of copies of its shadow graph $S$ in $G(n,p)$.
If $C$ is dense, we have $\Eb[X]=(1+o(1))\Eb[Y]=\Theta(n^{v(C)}p^{e(F)e(C)})$.
On the other hand, if $C$ is sparse, then we have $\Eb[X]=(1+o(1))p\Eb[Y]$ and $\Eb[Y]=\Theta(n^{v(C)}p^{2e(F)-1})$.

In order to match $\Eb[X]$ and $\Eb[Y]$, we select only a fraction $p$ of the sparse clean $F$-cycles in $G(n,p)$. To model this selection, we introduce the following `dummy edges'.

\begin{definition}\label{def:dgraph}
 A \emph{d-graph} $G=(V,U\cup D)$ is given by vertices $V$ and an edge set $E=U\cup D$, where the \emph{usual edges} $U$ are some subsets of $V$ of size $2$, and the \emph{dummy edges} $D$ are some sparse clean $F$-cycles with vertices in $V$.
 The d-graph $G=(V,U\cup D)$ is a \emph{clean d-cycle} if  
 \begin{enumerate}
  \item $(V,U)$ is the shadow graph of a dense clean $F$-cycle and $D=\emptyset$, or
  \item $(V,U)$ is the shadow graph of a sparse clean $F$-cycle $C$ and, \emph{additionally}, $D=\{C\}$, that is, the dummy edge corresponding to $C$ is present.
 \end{enumerate}
\end{definition}

This definition ensures that for any clean d-cycle $C$ of length $k$ we have $ke(F)$ edges.
Finally, the \emph{random d-graph} $G^*(n,p)$ has the vertices $[n]$, and each edge (including dummy edges) is included independently with probability~$p$. Note that we obtain $G(n,p)$ from $G^*(n,p)$ by discarding the dummy edges.

For d-graphs $G$ we also use $v(G),e(G)=|U\cup D|$, $c(G)$ as before, where $G$ is connected if $G$ without the dummy edges is. We also use the 1-density $d_1(G)$ from Equation \eqref{equ:1density_def} and the edge density $e(G)/v(G)$, which means that also strictly 1-balanced and strictly balanced d-graphs are defined.
The next result is a key argument in the proof.
\begin{lemma}\label{lem:cc_strictly_balanced}
 All clean d-cycles are strictly balanced.
\end{lemma}
\begin{proof}
Let $C$ be a dense clean $F$-cycle with copies $F_1,\dots,F_k$ and overlaps $v_1,\dots,v_k$ as given by the definition, and let $G$ be its shadow graph.
Then we have
$$\frac{e(G)}{v(G)}=\frac{ke(F)}{k(v(F)-1)}=d_1(F).$$
Now, assume that $G$ is not strictly balanced and let $S\subsetneqq G$ be strictly balanced with maximal density. 
Assume that $S$ does not cover one of the overlap vertices $v_i$.
Then we can split at $v_i$ to see that $S$ is a subgraph of the shadow graph of a clean $F$-path $P$ with copies $F'_1,\dots,F'_\ell$ of $F$, that is, $F'_i$ and $F'_{i+1}$ overlap in exactly one vertex for $i<\ell$ and $F'_i$, $F'_j$ have no overlap otherwise.
Since $S$ is connected, we may assume without loss of generality that $S$ covers all overlap vertices of $P$ (possibly $\ell=1$), and that, using $S_i=F'_i\cap S$, we have $v(S_i) \ge 2$ for all $i$. It follows that
\begin{align}\label{equ:1density_cleanpathbound}
 \frac{e(S)}{v(S)} < d_1(S)=\sum_i\frac{e(S_i)}{v(S)-1}=\sum_i\frac{v(S_i)-1}{\sum_{j}(v(S_{j})-1)}d_1(S_i)\le d_1(F)=\frac{e(G)}{v(G)}.
\end{align}
But if $S$ covers all overlaps, then using the parts $S_i=F_i\cap S$ we have $$\frac{e(S)}{v(S)}=\sum_i\frac{v(S_i)-1}{\sum_{j}(v(S_{j})-1)}d_1(S_i)<d_1(F)=\frac{e(G)}{v(G)}.$$
Now, let $C$ be a sparse clean $F$-cycle and $G$ the shadow graph including the dummy edge. Further, let $S\subset G$ be a proper sub-d-graph. If $S$ includes the dummy edge, then we have $v(S)=v(G)$ and thus the result follows.
Otherwise, $S$ is a proper subgraph of $G$ without the dummy edge.
But since we still have $e(G)/v(G)=d_1(F)$, the discussion above completes the proof.
\end{proof}

\section{Clean Cycles}\label{sec:clean_cycle_coupling}
We proceed to establish the results discussed in Section \ref{sec:ps:cc_coupling}.
In Section \ref{sec:ccc_overlaps}, we show that for both clean $F$-cycles and clean d-cycles, whp there are no pairs of distinct overlapping clean cycles.
In Section \ref{sec:ccc_coupling}, we use this result to obtain the desired coupling of the clean cycles.

\subsection{Overlaps}\label{sec:ccc_overlaps}
First, we turn to the overlaps of clean $F$-cycles of length at most $e(F)$.
\begin{lemma}\label{lem:cFc_overlapsH}
 Let $C_1,C_2$ be distinct clean $F$-cycles with $e(C_1), e(C_2) \le e(F)$. Assume that $C_1,C_2$ are not vertex-disjoint and let $X$ be the number of copies of $C_1\cup C_2$ in $\fgraph(n,\pi)$. Then we have $\mathbb E[X]=\mc O(n^{-1+2e(F)\varepsilon})$.
\end{lemma}
\begin{proof}
 Note that $C_1\cup C_2$ is an avoidable configuration as it contains two cycles, so Lemma \ref{lem:avoidable_configs} completes the bound.
\end{proof}
This shows that if $\varepsilon$ is chosen small enough, whp there are no overlapping clean $F$-inducing $F$-cycles up to length $e(F)$.
Next, we derive the analogous result for clean d-cycles.
\begin{lemma}\label{lem:cFc_overlapsG}
 Let $G_1,G_2$ be distinct clean d-cycles that are not vertex-disjoint and both correspond to clean $F$-cycles of length at most $e(F)$. Let $X$ be the number of copies of $G_1\cup G_2$ in $G^*(n,p)$. 
 Then we have $\mathbb E[X]=o(1)$ if $\varepsilon$ is chosen small enough.
\end{lemma}
\begin{proof}
 Let $C_i$ be the clean $F$-cycle corresponding to $G_i$.
 Further, let $S=G_1\cap G_2$ be the intersection graph, then we have
 $\mathbb E[X]=\mc O(n^{f(S)})$ with
 \begin{align*}
  f(S)&=v(C_1)+v(C_2)-v(S)-(r-1-\varepsilon)\left(e(C_1)+e(C_2)-\frac{e(S)}{e(F)}\right)\\
  &\leq(r-1)\frac{e(S)}{e(F)}-v(S)-n(C_1)+1-n(C_2)+1+2e(F)\varepsilon\\
  &=\left(\frac{e(S)}{v(S)d_1(F)}-1\right)v(S)+2e(F)\varepsilon<0,
 \end{align*}
 using Lemma \ref{lem:cc_strictly_balanced} (in any case, be it sparse or dense).
\end{proof}
\subsection{Coupling}\label{sec:ccc_coupling}
Now, we are ready to match the clean $F$-cycles in the random $F$-graph with the clean d-cycles in the random d-graph.
For this purpose, we will need the following deep and beautiful theorem from \cite{arratia1989}. Since we do not need the theorem in its full generality, we restrict to the special case in question and adapt the notation along the way.
\begin{theorem}\label{thm:poisson_tv}
 Let $X=(X_1,\dots,X_n)$ be Bernoulli and $B_1,\dots,B_n$ be subsets of $[n]$ such that $X_i$ is independent of $(X_j)_{j\not\in B_i}$.
 Let $Y=(Y_1,\dots,Y_n)$ be independent Poisson variables with parameters $\mathbb E[Y_i]=\mathbb E[X_i]$.
 Then the total variation distance of $X$ and $Y$ is at most
 \begin{equation}\label{eq:dtv}
  4\left(\sum_i\sum_{j\in B_i}\mathbb E[X_i]\mathbb E[X_j]+\sum_i\sum_{j\in B_i\setminus\{i\}}\mathbb E[X_iX_j]\right).
 \end{equation}
\end{theorem}
We apply this result twice.
For one, we use it for the clean $F$-cycles $\mathcal C_1$ corresponding to the clean d-cycles of length at most $e(F)$ contained in $G^*(n,p)$, but also for the clean $F$-cycles $\mathcal C_2$ of length at most $e(F)$ in $\fgraph(n,\pi)$. Combining the two resulting couplings yields the following desired coupling of the clean cycles.
\begin{proposition}\label{prop:cFc_coupling}
 If $\varepsilon>0$ is chosen small enough, there exists a coupling of $\mathcal C_1$ and $\mathcal C_2$ such that~$\mathcal C_1=\mathcal C_2$ whp.
\end{proposition}
\begin{proof}
 Denote by $\mathscr{C}$ the set of all potential copies of clean $F$-cycles of length at most $e(F)$ in $H_F(n,\pi)$. Each $C\in \mathscr{C}$ corresponds bijectively to a potential clean d-cycle in $G^*(n,p)$, and thus we will use $\mathscr{C}$ as an index set for both clean $F$-cycles and clean d-cycles by a slight abuse of notation.

 For $C\in\mathscr{C}$ let $X_C$ be the indicator for the event that $C$ is in $\fgraph(n,\pi)$.
 Furthermore, let $B_C$ be the set of all clean $F$-cycles that are not vertex disjoint from $C$.
 Note that for $C'\in B_C$ we have $\mathbb E[X_CX_{C'}]\ge\mathbb E[X_C]\mathbb E[X_{C'}]$, thus the term in \eqref{eq:dtv} is upper bounded by
 \[
  8\sum_{C\in\mathscr{C}}\ \sum_{C'\in B_C\setminus\{C\}} \mathbb{E}[X_C X_{C'}]  + 4 \sum_{C\in\mathscr{C}} \mathbb{E}[X_C]^2.
 \]
 This is $o(1)$: Indeed, collecting terms in the double-sum according to the isomorphism class of $C\cup C'$ (of which there are $\mc O(1)$ many) yields a bound of $\mc O(n^{-1+2e(F)\varepsilon})=o(1)$ by Lemma~\ref{lem:cFc_overlapsH}. Similarly, splitting up the second sum according to the isomorphism class of $C$ yields $\mc O\big(n^{2e(F)(\varepsilon - 2(r-1)}\big)=o(1)$, analagous to the proof of Lemma~\ref{lem:avoidable_configs}. Hence there is a coupling of of $(X_C)_{C\in\mathscr{C}}$ and independent Poisson variables $(Y_C)_{C\in\mathscr{C}}$ with parameters $\mathbb E[X_C]$ such that $(X_C)_C=(Y_C)_C$ whp. 
 
 Relying on Lemma~\ref{lem:cFc_overlapsG}, we proceed analogously for $G^*(n,p)$ to arrive at a coupling of the indicators $(X'_C)_C$ and independent Poisson variables $(Y'_C)_C$ with parameters $\mathbb E[X'_C]$ such that $(X'_C)_C=(Y'_C)_C$ whp.
 
 Now, note that $\mathbb E[X_C]=(1-n^{-\delta})^{e(C)}\mathbb E[X'_C]$. Therefore,
 \[
  d_{TV}\big((Y_C)_C,(Y'_C)_C\big) 
  \leq \sum_{C\in\mathscr{C}} d_{TV}(Y_C,Y'_C)
  \leq \sum_{C\in\mathscr{C}} \left((1-n^{-\delta})^{-e(C)}-1\right)\mathbb E[X_C]
  =\mc O\big(n^{-\delta + \varepsilon e(F)}\big)
 \]
 by Lemma~\ref{lem:avoidable_configs} after collecting summands with isomorphic $C$ and noting that $n(C)=1$.
 Thus for $\varepsilon < \delta/e(F)$, we can couple $(Y_C)_C$ with $(Y'_C)_C$ such that $(Y_C)_C=(Y'_C)_C$ whp.
\end{proof}

%
%
\section{The Riordan-Heckel Coupling}\label{sec:hypergraph_coupling}
In this section, we complete the proof of Theorem \ref{thm:main_coupling} by formalizing the discussion in Section \ref{sec:ps:heckel_riordan_coupling}.
In particular, we present the second part of the coupling, which is morally unchanged compared to \cite{heckel2020random}. Details can be found in Section \ref{sec:coupling_alg}.
Then, in Section \ref{sec:coupling_works}, we state the main result of this section, namely that the coupling succeeds whp. 
Sections \ref{sec:cond_prob} to \ref{sec:prop_coupling_works_proof} are dedicated to its proof.
Finally, in Section \ref{sec:thm_main_coupling_proof}, we establish Theorem \ref{thm:main_coupling} and in Section \ref{sec:thm_main_threshold_proof} we derive Theorem \ref{thm:main_threshold}.
\subsection{The Coupling Algorithm}\label{sec:coupling_alg}
We state the coupling from \cite{heckel2020random} for the general case.
So, as in \cite{heckel2020random} we fix some arbitrary order $F_1,\dots,F_M$ of the $F$-edges, where $M=\frac{r!}{\aut(F)}\binom{n}{r}$ is the total number of $F$-edges.
Now, we start the coupling of $G=G^*(n,p)$ and $H=\fgraph(n,\pi)$ by coupling $\mathcal C_1$ and $\mathcal C_2$ using Proposition \ref{prop:cFc_coupling}.
On the unlikely event that $\mathcal C_1\neq\mathcal C_2$, we say that the coupling fails.

We proceed with the coupling given $\mathcal C_1$ and $\mathcal C_2$, i.e.~specific (but any specific) choices of clean $F$-cycles for $H$, which are the clean d-cycles in $G$.
In this second part and in the given order, we iteratively decide for each $F$-edge if it is contained in $H$ and sometimes also if it is contained in $G$.

In step $j$, we consider the conditional probability $\pi_j$ that $F_j$ is included in $G$ and the conditional probability $\pi'_j$ that $F_j$ is an $F$-edge of $H$, given all available information so far -- in particular, this includes the given set $\mathcal C_1$ for both $G$ and $H$.
To be thorough, for $\pi'_j=\pi_j=0$, we decide to exclude $F_j$ from both $G$ and $H$.

Otherwise, for $\pi'_j\le\pi_j$, we flip a coin with success probability $\pi'_j/\pi_j$.
In case of success, we decide if $G$ contains 
$F_j$ or not (with probability $\pi_j$) and reflect this decision in $H$.
Otherwise, we do not make a decision for $G$ and do not include $F_j$ in $H$. 

For $\pi'_j>\pi_j$, we decide if we include $F_j$ in $H$ (with probability $\pi'_j$), but do not make a decision for $G$.
If $F_j$ is included in $H$ in this case, we say that the coupling fails.

Finally, we note that all decisions for $H$ have been made.
On the other end, we use the conditional distribution of $G=G^*(n,p)$, given all (graph-related) information so far, to generate~$G$.
Note that both $G$ and $H$ now have the correct distributions.

\subsection{Main Result}\label{sec:coupling_works}
Let $\mathcal B_1$ be the event that the maximum degree of $\fgraph(n,\pi)$ is at least $\Delta=n^{\eps}+\ln(n)n^{\eps/2}$. 
Let $\mathcal B_2$ be the event that $\fgraph(n,\pi)$ contains an avoidable configuration with at most $2e(F)^2$ $F$-edges.
Let $\mathcal B_3$ be the event that the coupling of $\mathcal C_1$ and $\mathcal C_2$ fails.
Finally, let $\mathcal B=\mathcal B_1\cup\mathcal B_2\cup\mathcal B_3$.
Note that $\Pb (\mathcal B_1)=o(1)$ by the Chernoff bound, $\Pb(\mathcal B_2)=o(1)$ by Lemma \ref{lem:avoidable_configs} (for small enough $\varepsilon>0$) and $\Pb(\mathcal B_3)=o(1)$ by Proposition \ref{prop:cFc_coupling}, thus whp $\mathcal B$ does not hold. Also, recall from the proof of Lemma \ref{lem:cFc_overlapsH} that overlapping clean $F$-cycles are avoidable configurations, so the clean $F$-cycles $\mathcal C_1=\mathcal C_2$ are pairwise vertex-disjoint unless $\mathcal B$ holds.
\begin{proposition}\label{prop:coupling_works}
If the coupling in Section \ref{sec:coupling_alg} fails, then $\mathcal B$ holds.
\end{proposition}
For the proof, we proceed with the coupling until it fails for the first time. If it fails when coupling $\mathcal C_1$ and $\mathcal C_2$, then $\mathcal B_3$ holds and we are done.
Otherwise, we proceed to fix $\mathcal C_1$ and thereby also $\mathcal C_2=\mathcal C_1$.
If these cycles are not pairwise vertex disjoint, then $\mathcal B_2$ holds and we are done.
Otherwise, we enter the second part of the coupling, which will be discussed in the remainder of the 
paper.
In Section \ref{sec:cond_prob}, we entirely redo the argumentation in \cite{heckel2020random} to obtain accessible bounds for the conditional probabilities.
Then, we have to control the four contributions to the bounds.
In Section \ref{sec:bad_edges}, we discuss $F$-edges that break the coupling.
In Section \ref{sec:bad_cFc}, we discuss clean $F$-cycles that break the coupling.
In Section \ref{sec:good_edges}, we discuss the contribution of excluded $F$-edges to the bound.
In Section \ref{sec:good_cFc}, we discuss the contribution of excluded clean $F$-cycles to the bound. Finally, in Section~\ref{sec:prop_coupling_works_proof} we complete the proof of Proposition~\ref{prop:coupling_works}.

\subsection{Conditional Probabilities}\label{sec:cond_prob}
We follow \cite{riordan2022random} with the adjustments from \cite{heckel2020random}.
Assume that the coupling fails for the first time in step $j\geq 1$.
Note that prior to $j$ we have chosen to include an $F$-edge $F_i$ in $\fgraph(n,\pi)$ exactly if we chose to include it in $G^*(n,p)$. Let $Y\subset[j-1]$ be the steps where we made these decisions.
For all steps $i\in[j-1]\setminus Y$ we decided to exclude $F_i$ from $\fgraph(n,\pi)$, but only for a subset $N'\subset[j-1]\setminus Y$ of steps $i\in N'$ we decided to exclude $F_i$ from $G^*(n,p)$.
For the remaining steps $[j-1]\setminus(Y\cup N')$, we made no decision regarding the containment of $F_i$ in $G^*(n,p)$.

Further, let $H_0$
be the $F$-graph on vertices $[n]$ with $F$-edges $F_i$ for $i\in Y$ and clean $F$-cycles $\mathcal C_1$.
If $H_0\cup F_j$ has maximum degree at least $\Delta$, then $\mathcal B_1$ holds and we are done.
Further, enumerate all (potential) clean $F$-cycles except those in $\mathcal C_1$ with $\mc C_1^c=\{F_{-1},\dots,F_{-\ell}\}$, set $N''=\{-1,\dots,-\ell\}$ and let $N=N'\cup N''$. 

Before we turn to the core of the proof, we want to build some intuition. 
Note that $\pi'_j=1$ exactly if $F_j$ is in some cycle in $\mathcal C_1$, in which case we also have $\pi_j=1$ and the coupling algorithm did choose to include $F_j$ almost surely in both models. 
Further, note that $\pi'_j=0$ exactly if $F_j$ is in some cycle $C\in \mc C_1^c$ and all other $F$-edges $(F_s)_{s\in S}$ of $C$ were previously included, i.e.~$s\in Y\cap[j-1]$, or are included in some cycle $C'$ in $\mc C_1$, i.e.~$F_s\in C'$.
If $C$ is a dense clean $F$-cycle, then we also have $\pi_j=0$ and $F_j$
is excluded in $G^*(n,p)$.
In the \emph{new case} that $C$ is a sparse clean $F$-cycle, we do have $\pi_j>0$ because both $F$-edges might still be included in $G^*(n,p)$ just as long as the corresponding dummy edge is not. However, in this case our coupling flips the coin with probability $\pi'_j/\pi_j=0$, meaning that $F_j$ will be excluded in $\fgraph(n,\pi)$ and not tested in $G^*(n,p)$ almost surely.

Now we turn to the bound for $\pi'_j$.
Let $H'$ be the random $F$-graph on vertices $[n]$ where each $F$-edge of $H_0$ is included, and every
other $F$-edge is included independently with probability $\pi$.
Let $\mathcal L_2$ be the event that the cycles in $\mc C_1^c$ are not included, then we have
\begin{align}\label{equ:condprob_boundH}
\pi'_j=\mathbb P(F_j\textrm{ in }H'|\mathcal L_2)\le\mathbb P(F_j\textrm{ in }H')=\pi
\end{align}
unless $\pi'_j=1$ (cf.~above), since $\mathcal L_2$ is a down-set in the product probability space corresponding to $H'$, while the event $\{F_j\in H'\}$ is an up-set.

Next, we turn to the bound for $\pi_j$, as detailed in \cite{riordan2022random} and partially inspired by the proof of Janson's inequality and a variation suggested by Lutz Warnke \cite{riordan2015janson}.
The following arguments are based on $G^*(n,p)$, so the discussion always involves dummy edges as well unless explicitly stated otherwise.
Let $E_i$ be the edge set (including any dummy edges) of $F_i$ for all $i$ (positive and negative) and let $R$ be the edge set of the shadow graph of $H_0$.

As for $H'$, let $G'$ be the random d-graph where $R$ is included and every other edge is included independently with probability $p$.
For $i\le j$ let $E'_i=E_i\setminus R$.
What we know about $G^*(n,p)$ is exactly that $R$ is included, but none of the sets $E'_i$, $i \in N$, of edges are. 
For one, note that this includes the cycles with negative indices.
Also, note that $e\in E'_i$ may well be included, be it dummy edge or not, but not all of $E'_i$.
Now, let $A'_i$ be the event that all edges $E'_i$ are present in $G'$. Then in $G'$ we have $$\pi_j=\mathbb P\left(A'_j\middle|\bigcap_{i \in N}(A_i')^{\mathrm c}\right).$$
Now, we consider which events $A'_i$ are independent of $A'_j$, thus let
$$D_0=\bigcap_{i\in N:E'_i\cap E'_j=\emptyset}(A_i')^{\mathrm c}\quad\textrm{ and }\quad
D_1=\bigcap_{i\in N:E'_i\cap E'_j\neq\emptyset}(A_i')^{\mathrm c}.$$
Then we have 
\begin{align*}
 \pi_j&=\mathbb P(A'_j|D_0\cap D_1)=\frac{\mathbb P(A'_j\cap D_0\cap D_1)}{\mathbb P(D_0\cap D_1)}
 \ge\frac{\mathbb P(A'_j\cap D_0\cap D_1)}{\mathbb P(D_0)}\\
 &=\mathbb P(A'_j\cap D_1|D_0)=\mathbb P(A'_j|D_0)-\mathbb P(A'_j\cap D^{\mathrm c}_1|D_0).
\end{align*}
Notice that we can use independence for the former and the Harris inequality for the latter, yielding
\[
 \pi_j\ge\mathbb P(A'_j)-\mathbb P(A'_j\cap D_1^{\mathrm c}).
\]
For brevity, we denote the index set of $D_1$ by
\[
 N_1=N_{j,1}=\{i\in N:E'_i\cap E'_j\neq\emptyset\}
\]
and hence $D_1^{\mathrm c}=\bigcup_{i\in N_1}A'_i$. 
Using the distributive property and the union bound gives
\[
 \pi_j\ge\mathbb P(A'_j)-\sum_{i\in N_1}\mathbb P(A'_j\cap A'_i).
\]
Finally, we obtain the desired bound
\begin{align}\label{equ:condprob_boundG}
 \pi_j\ge p^{|E'_j|}-\sum_{i\in N_1}p^{|E'_i\cup E'_j|}\ge(1-Q)p^{e(F)},
\end{align}
where the relative error $Q$ is given by
\begin{align}\label{equ:Q_def}
 Q=Q_j=\sum_{i\in N_1}p^{|E'_i\setminus E'_j|}=\sum_{i\in N_1}p^{|E_i\setminus(E_j\cup R)|}.
\end{align}
Now, we split $Q=Q_{\mathrm{cb}}+Q_{\mathrm{cg}}+Q_{\mathrm{eb}}+Q_{\mathrm{eg}}$ for the discussion in the subsequent four sections, 
as follows:
\begin{itemize}
 \item $Q_{\mathrm{cb}}$ is the contribution coming from \emph{bad cycles}, i.e. $i<0$ with $E_i\subset E_j\cup R$
 \item $Q_{\mathrm{cg}}$ is the contribution coming from \emph{good cycles}, i.e. $i<0$ with $E_i\not\subset E_j\cup R$
 \item $Q_{\mathrm{eb}}$ is the contribution coming from \emph{bad edges}, i.e. $i>0$ with $E_i\subset E_j\cup R$
 \item $Q_{\mathrm{eg}}$ is the contribution coming from \emph{good edges}, i.e. $i>0$ with $E_i\not \subset E_j\cup R$.
\end{itemize}

\subsection{Bad Edges}\label{sec:bad_edges}
In this section, we argue that $Q_{\mathrm{eb}}=0$ unless $\mathcal B$ holds, in which case we are done.
For $i\in N_1$ contributing to $Q_{\mathrm{eb}}$, we have $E_i\subset E_j\cup R$.
With $i\in N_1$, we know that $F_i$ is not in $H_0\cup F_j$ and thus also not in $\fgraph(n,\pi)$.
On the other hand, $H_0\cup F_j$ is contained in $\fgraph(n,\pi)$ and hence $E_i\subset E_j\cup R$ implies that $F_i$ is in the $F$-graph of the shadow graph of $\fgraph(n,\pi)$, meaning that $F_i$ is induced.
By Lemma \ref{lem:induced}, there exists an avoidable configuration or a clean $F$-cycle inducing $F_i$. If the former is true, then $\mathcal B_2$ holds and we are done.
Otherwise, $F_i$ needs to be induced by a cycle in $\mathcal C_1=\mathcal C_2$. But this implies that $\pi_i=1$, which in turn implies that $i\not\in N_1$, a contradiction.

\subsection{Bad Cycles}\label{sec:bad_cFc}
In this section, we argue that $Q_{\mathrm{cb}}=0$ unless $\mathcal B$ holds, in which case we are done.
For $i\in N_1$ contributing to $Q_{\mathrm{cb}}$ (recall that this means that $i<0$ and $F_i$ is a clean cycle), we have $E_i\subset E_j\cup R$.
With $i\in N_1$, we know that $F_i$ is not in $H_0\cup F_j$ and thus also not in $\fgraph(n,\pi)$.
On the other hand, $H_0\cup F_j$ is contained in $\fgraph(n,\pi)$ and hence $E_i\subset E_j\cup R$ suggests that $F_i$ is in the $F$-graph of the shadow graph of $\fgraph(n,\pi)$, meaning that some $F$-edges of $F_i$ have to be induced.
As in Section \ref{sec:bad_edges} and using Lemma \ref{lem:induced}, if avoidable configurations are involved, then $\mathcal B_2$ holds and we are done.

Otherwise, also as in Section \ref{sec:bad_edges}, for each induced $F$-edge $F'$ in $F_i$ there exists a cycle $C(F')$ in $\mathcal C_1$ that induces $F'$.
But with $\mathcal I$ denoting the set of induced $F$-edges in $F_i$, it is immediate that 
$$(F_i\setminus\mathcal I)\cup\bigcup_{F'\in\mathcal I}C(F')$$
is both in $\fgraph(n,\pi)$ and an avoidable configuration, so $\mathcal B_2$ holds and we are done.

\subsection{Good Edges}\label{sec:good_edges}
In this section, we establish a bound for $Q_{\mathrm{eg}}$ unless $\mathcal B_1$ holds, in which case we are done. The argument is based on \cite{riordan2022random}, the approach in \cite{burghart2024} is more explicit than the argument below.
We obtain the bound by bounding the number of contributions for a given exponent. This is achieved by bounding the number of possible vertex sets of the corresponding $F_i$.

For $i\in N_1$ contributing to $Q_{\mathrm{eg}}$, we have $E_i\setminus(E_j\cup R)\neq\emptyset$.
On the other hand, by the definition of $N_1$, we have $(E_i\cap E_j)\setminus R\neq\emptyset$.
To obtain a bound on the number of vertex sets, we consider the graph intersection $S$ given by the vertex set of $F_i$ and the edges from $E_j\cup R$ that are in $F_i$, i.e. $S=(V(F_i),E_i\cap (E_j\cup R))$.
Hence, the contribution can be written as $p^{|E_i\setminus(E_j\cup R)|}=p^{e(F)-e(S)}$.
Now, we count the number of contributions given $e(S)$ and $c(S)$.
We first count the number of vertex sets that lead to such $S$; to each vertex sets there are then at most $r!/\aut(F)$-many possible $F_i$.
For the vertex set, we choose an anchor for each component, respecting that at least one component has to take one of the vertices in $E_j$, which gives $rn^{c(S)-1}$.
Then, for each component individually, given the chosen anchor vertex, the remaining vertices of the component can be picked step-by-step from the neighbours (in $E_j \cup R$) of the vertices picked so far, yielding the bound $r(r-1)\Delta$ for each such step.
Thus, the total number of vertex sets is bounded by $\mc O(n^{c(S)-1}\Delta^{r-c(S)})=\mc O(n^{c(S)-1+(r-c(S))\varepsilon})$.
Thus, the total contribution for given $e(S)$ and $c(S)$ to $Q_{\mathrm{eg}}$ is bounded by
$\mc O(n^{f(S)})$, where
\begin{align*}
 f(S)&=c(S)-1+(r-c(S))\eps-\frac{r-1-\eps}{e(F)}(e(F)-e(S))=f_1(S)+f_2(S)\eps,\\
 f_1(S)&=\frac{e(S)}{d_1(F)}-(v(S)-c(S)),\quad
 f_2(S)=v(S)-c(S)+1-\frac{e(S)}{e(F)}.
\end{align*}
Note that $f_1$ is additive over disjoint graphs and zero for isolated vertices. For each non-trivial component $S'$ of $S$ we have $d_1(S')<d_1(F)$, implying $f_1(S')=\left(\frac{d_1(S')}{d_1(F)}-1\right)(v(S')-1)<0$.  
This shows that $f(S)<0$ for any of these subgraphs $S$ of $F$ if $\varepsilon$ is small enough, yielding that the contribution from $Q_{\mathrm{eg}}$ is $o(1)$.

\subsection{Good Cycles}\label{sec:good_cFc}
In this section, we establish a bound for $Q_{\mathrm{cg}}$ unless $\mathcal B_1$ holds, in which case we are done. Then we complete the proof of Proposition \ref{prop:coupling_works}.
In principle, we follow \cite{heckel2020random}, however we need a novel approach to deal with general strictly 1-balanced graphs.
Now, additionally to the number of edges and the number of components of the graph intersection,
we also fix the clean $F$-cycle $C$ and focus on the $i\in N_1$ with $F_i$ being a copy of $C$.
As before, we bound the number of possible vertex sets, which suffices to obtain a bound on $Q_{\mathrm{cg}}$.

For $i\in N_1$ contributing to $Q_{\mathrm{cg}}$ where $F_i$ is a copy of $C$, we have $E_i\setminus(E_j\cup R)\neq\emptyset$.
On the other hand, by the definition of $N_1$, we have $(E_i\cap E_j)\setminus R\neq\emptyset$.
Let $S$ be given by the vertex set of $F_i$ and the edges from $E_j\cup R$ on these vertices, so the contribution can be written as $p^{|E_i\setminus(E_j\cup R)|}=p^{e(G)-e(S)}$, where $G$ is the clean d-cycle corresponding to $C$.
Now, we count the number of vertex sets given $e(S)$, $c(S)$ and the type $C$.
As in Section~\ref{sec:good_edges}, we choose an anchor for each component, respecting that one component has to take one of the vertices in $E_j$, which gives $\mc O(n^{c(S)-1})$. 
Then, for each component individually, given the anchor vertex we may choose the remaining vertices step-by-step from the neighbours of the vertices we have chosen so far.
Thus, the total is bounded by $\mc O(n^{c(S)-1}\Delta^{v(S)-c(S)})=\mc O(n^{c(S)-1+(v(S)-c(S))\varepsilon})$.
Thus, the total contribution for given $e(S)$ and $c(S)$ to $Q_{\mathrm{eg}}$ is bounded by
$O(n^{g(S)})$, where
\begin{align*}
g(S)&=c(S)-1+(v(S)-c(S))\eps-\frac{r-1-\eps}{e(F)}(e(G)-e(S))=g_1(S)+g_2(S)\eps,\\
g_1(S)&=\frac{e(S)}{d_1(F)}-(v(S)-c(S))-1,\quad
g_2(S)=v(S)-c(S)+\frac{e(G)-e(S)}{e(F)}.
\end{align*}
Now, we show that $g_1(S)<0$ for all $S$.
Note that $g_1(S)=f_1(S)-1$ with $f_1$ from Section \ref{sec:good_edges} 
and recall that $f_1$ is additive over disjoint graphs. 
Furthermore, as in Equation~\eqref{equ:1density_cleanpathbound}, we have $d_1(S')\le d_1(F)$ for any component $S'$ of $S$ which does not cover all overlap vertices of the clean cycle $F_i$, which implies $f_1(S')\le 0$. It is thus sufficient to show $g_1(S')<0$ for a component covering all overlap vertices (and such $S'$ is unique in $S$, if it exists). For such $S'$ we have 
\[
 g_1(S)=\left(\frac{e(S)}{v(S)d_1(F)}-1\right)v(S) < \left(\frac{d(F)}{d_1(F)}-1\right)v(S_1) <0
\]
by Lemma~\ref{lem:cc_strictly_balanced}, completing the proof that $g_1(S)<0$.

\subsection{Proof of Proposition \ref{prop:coupling_works}} \label{sec:prop_coupling_works_proof}
As indicated throughout Sections~\ref{sec:good_edges} and \ref{sec:good_cFc}, we now choose any $\delta>0$ such that
\[
 \max\left(\max_S f_1(S),\max_S g_1(S)\right)<-2\delta<0
\]
and $\varepsilon>0$ sufficiently small as required throughout the proof, e.g. by $\Pb(\mathcal{B}_2)=o(1)$, the proof of Proposition~\ref{prop:cFc_coupling}, and in particular such that $f(S),g(S)<-\delta$ over all $S$.
Then, after collecting the contributions from all four cases, we have $$\pi_j\ge(1-n^{-\delta})p^{e(F)}=\pi\ge\pi'_j$$
using Equations \eqref{equ:condprob_boundH}, \eqref{equ:condprob_boundG} and \eqref{equ:Q_def} and thus the coupling cannot fail in step $j$.
This completes the proof of Proposition \ref{prop:coupling_works}.

\subsection{Proof of Theorem \ref{thm:main_coupling}}\label{sec:thm_main_coupling_proof}
As discussed in Section \ref{sec:coupling_works}, we know that whp the event $\mc B$ does not hold, so Proposition \ref{prop:coupling_works} shows that the coupling from Section \ref{sec:coupling_alg} does not fail whp. But if it does not fail, then every $F$-edge of $\fgraph(n,\pi)$ is contained as a subgraph in $G^*(n,p)$ and thus also in $G(n,p)$.

\subsection{Proof of Theorem \ref{thm:main_threshold}}\label{sec:thm_main_threshold_proof}
We combine the coupling of $H_r(n,\pi')$ and $\fgraph(n,\pi)$ from Section \ref{sec:main_result} with the coupling of $\fgraph(n,\pi)$ and $G(n,p)$ from Theorem \ref{thm:main_coupling}. Then Theorem \ref{thm:shamirs_problem} applied to $H_r(n,\pi)$
yields that $G(n,p)$ contains an $F$-factor whp, for all $p\ge(1+\eps)p^*$ and $\eps>0$.
Combined with Theorem \ref{thm:threshold_COV}, this completes the proof.

\bibliographystyle{plain}
\bibliography{Strictly_1balanced_factors_v2}

\begin{thebibliography}{10}

\bibitem{Ajtai1987FirstOccurrence}
Miklós Ajtai, János Komlós, and Endre Szemerédi.
\newblock The first occurrence of {H}amilton cycles in random graphs.
\newblock {\em Annals of Discrete Mathematics}, 27:173--178, 1985.

\bibitem{alon1993threshold}
Noga Alon and Raphael Yuster.
\newblock Threshold functions for {$H$}-factors.
\newblock {\em Combin. Probab. Comput.}, 2(2):137--144, 1993.

\bibitem{arratia1989}
Richard Arratia, Larry Goldstein, and Louis Gordon.
\newblock Two moments suffice for {P}oisson approximations: the {C}hen-{S}tein
  method.
\newblock {\em The Annals of Probability}, pages 9--25, 1989.

\bibitem{bollobas1985}
B\'{e}la Bollob\'{a}s and Andrew Thomason.
\newblock Random graphs of small order.
\newblock In {\em Random graphs '83 ({P}ozna\'{n}, 1983)}, volume 118 of {\em
  North-Holland Math. Stud.}, pages 47--97. North-Holland, Amsterdam, 1985.

\bibitem{burghart2024}
Fabian Burghart, Marc Kaufmann, Noela Müller, and Matija Pasch.
\newblock The hitting time of nice factors, 2024.

\bibitem{diaz2023spanning}
Alberto~Espuny D{\'\i}az and Yury Person.
\newblock Spanning-cycles in random graphs.
\newblock {\em Combinatorics, Probability and Computing}, 32(5):833--850, 2023.

\bibitem{dudek2010loose}
Andrzej Dudek and Alan Frieze.
\newblock Loose {H}amilton cycles in random uniform hypergraphs.
\newblock {\em arXiv preprint arXiv:1006.1909}, 2010.

\bibitem{dudek2012optimal}
Andrzej Dudek, Alan Frieze, Po-Shen Loh, and Shelley Speiss.
\newblock Optimal divisibility conditions for loose {H}amilton cycles in random
  hypergraphs.
\newblock {\em the electronic journal of combinatorics}, pages P44--P44, 2012.

\bibitem{erdos1966existence}
P{\'a}l Erd\H{o}s and Alfr\'{e}d R\'{e}nyi.
\newblock On the existence of a factor of degree one of a connected random
  graph.
\newblock {\em Acta Math. Acad. Sci. Hungar.}, 17:359--368, 1966.

\bibitem{frieze2010loose}
Alan Frieze.
\newblock Loose {H}amilton cycles in random 3-uniform hypergraphs.
\newblock {\em The Electronic Journal of Combinatorics}, pages N28--N28, 2010.

\bibitem{frieze2020}
Alan Frieze.
\newblock A note on spanning {$K_r$}-cycles in random graphs.
\newblock {\em AIMS Mathematics}, 5(5):4849--4852, 2020.

\bibitem{gerke2015nonvertex}
Stefanie Gerke and Andrew McDowell.
\newblock Nonvertex-balanced factors in random graphs.
\newblock {\em Journal of Graph Theory}, 78(4):269--286, 2015.

\bibitem{heckel2020random}
Annika Heckel.
\newblock Random triangles in random graphs.
\newblock {\em Random Structures \& Algorithms}, 59(4):616--621, 2021.

\bibitem{heckel2024hitting}
Annika Heckel, Marc Kaufmann, Noela Müller, and Matija Pasch.
\newblock The hitting time of clique factors.
\newblock {\em Random Structures \& Algorithms}, 65(2):275--312, 2024.

\bibitem{janson2000random}
Svante Janson, Tomasz {\L}uczak, and Andrzej Ruci\'nski.
\newblock {\em Random graphs}.
\newblock Wiley-Interscience Series in Discrete Mathematics and Optimization.
  Wiley-Interscience, New York, 2000.

\bibitem{johansson2008factors}
Anders Johansson, Jeff Kahn, and Van Vu.
\newblock Factors in random graphs.
\newblock {\em Random Structures \& Algorithms}, 33(1):1--28, 2008.

\bibitem{kahn2022hitting}
Jeff Kahn.
\newblock Hitting times for {S}hamir's problem.
\newblock {\em Trans. Amer. Math. Soc.}, 375(1):627--668, 2022.

\bibitem{kahn2023asymptotics}
Jeff Kahn.
\newblock Asymptotics for {S}hamir's problem.
\newblock {\em Advances in Mathematics}, 422:109019, 2023.

\bibitem{kahn2007thresholds}
Jeff Kahn and Gil Kalai.
\newblock Thresholds and expectation thresholds.
\newblock {\em Combinatorics, Probability and Computing}, 16(3):495--502, 2007.

\bibitem{kim2003}
Jeong~Han Kim.
\newblock Perfect matchings in random uniform hypergraphs.
\newblock {\em Random Struct. Algorithms}, 23(2):111--132, 2003.

\bibitem{Krivelevich1997TriangleFI}
Michael Krivelevich.
\newblock Triangle factors in random graphs.
\newblock {\em Combinatorics, Probability and Computing}, 6:337 -- 347, 1997.

\bibitem{park2024proof}
Jinyoung Park and Huy Pham.
\newblock A proof of the {K}ahn--{K}alai conjecture.
\newblock {\em Journal of the American Mathematical Society}, 37(1):235--243,
  2024.

\bibitem{Poole2015StrengthConnectedness}
David Poole.
\newblock On the strength of connectedness of a random hypergraph.
\newblock {\em The Electronic Journal of Combinatorics}, 22(1):P1.69, 2015.

\bibitem{riordan2022random}
Oliver Riordan.
\newblock Random cliques in random graphs and sharp thresholds for
  {$F$}-factors.
\newblock {\em Random Structures \& Algorithms}, 61(4):619--637, 2022.

\bibitem{riordan2015janson}
Oliver Riordan and Lutz Warnke.
\newblock The {J}anson inequalities for general up-sets.
\newblock {\em Random Structures \& Algorithms}, 46(2):391--395, 2015.

\bibitem{ross2011}
N.~Ross.
\newblock Fundamentals of {S}tein's method.
\newblock {\em Probab. Surv.}, 8:210--293, 2011.

\bibitem{rucinski1992matching}
Andrzej Ruci\'{n}ski.
\newblock Matching and covering the vertices of a random graph by copies of a
  given graph.
\newblock {\em Discrete Math.}, 105(1-3):185--197, 1992.

\bibitem{spencer1990threshold}
Joel Spencer.
\newblock Threshold functions for extension statements.
\newblock {\em Journal of Combinatorial Theory Series A}, 53(2):286--305, 1990.

\bibitem{LuczakRucinski1987}
Tomasz Łuczak and Andrzej Ruciński.
\newblock Tree-matchings in random graphs.
\newblock Research Report 86-87, Rheinische Friedrich-Wilhelms-Universität
  Bonn, 1987.

\bibitem{LuczakRucinski1991}
Tomasz Łuczak and Andrzej Ruciński.
\newblock Tree-matchings in graph processes.
\newblock {\em SIAM Journal on Discrete Mathematics}, 4(1):107--120, 1991.

\end{thebibliography}

\end{document}